\documentclass{amsart}
\usepackage{amsmath,amsfonts,amssymb,stmaryrd}
\usepackage{graphicx,psfrag,subfigure}

\newtheorem{theorem}{Theorem}[section]
\newtheorem{corollary}{Corollary}

\newtheorem{lemma}[theorem]{Lemma}

\theoremstyle{definition}

\numberwithin{equation}{section}

\begin{document}

\title[ A simple proof of a theorem of sensitivity ]{ A simple proof of a theorem of sensitivity}

\author[Jorge Iglesias and Aldo Portela]{}


%
\address{J. Iglesias, Universidad de La Rep\'ublica. Facultad de Ingenieria. IMERL. Julio
Herrera y Reissig 565. C.P. 11300. Montevideo, Uruguay}
\email{jorgei@fing.edu.uy }

\address{A. Portela, Universidad de La Rep\'ublica. Facultad de Ingenieria. IMERL. Julio
Herrera y Reissig 565. C.P. 11300. Montevideo, Uruguay }
\email{aldo@fing.edu.uy }

\maketitle

\centerline{\scshape  Jorge Iglesias$^*$ and Aldo Portela$^*$}
\medskip
{\footnotesize
 \centerline{Universidad de La Rep\'ublica. Facultad de Ingenieria. IMERL}
   \centerline{ Julio Herrera y Reissig 565. C.P. 11300}
   \centerline{ Montevideo, Uruguay}}

\bigskip

 \centerline{(Communicated by )}

\begin{abstract} We prove that every transitive and non minimal semigroup with dense minimal points is sensitive. When the system is almost open, we obtain a generalization of this result.

\end{abstract}

\section{Introduction.}


 The
dynamical system that we will consider is formally defined as a triplet $(\mathcal{S},X,\Phi)$ where $\mathcal{S}$ is a topological
semigroup and $\Phi:\mathcal{S} \times X \to X$ is a continuous function with $\Phi (  s_1, \Phi(s_2 ,x))= \Phi(s_1 s_2 ,x)$ for all $s_1,s_2 \in \mathcal{S}$ and for all $x \in X$ and $X$ is a metric space. The map $\Phi$ is called an action of $\mathcal{S}$ on $X$. We will denote $\Phi (s,x)$ by $\Phi_s (x)$. Devaney defines a function to be chaotic if it satisfies the following three conditions: transitivity, having dense set of
periodic points and sensitive dependence on initial conditions.
 In \cite{3} it was proved that the two first conditions imply the last one. This result was generalized in \cite{aak}, by changing density of periodic points by density of
 minimal points. In \cite{g} this result was generalized when $\mathcal{S}$ is a group, in \cite{km} was generalized to $C$-semi groups and in \cite{d} for a continuous semi-flow and $X$ being a Polish space. The goal of this paper is to generalize the result in \cite{km}, giving a very simple proof. In \cite{grss} it is possible to find other results about sensitivity.

 \subsection{Basic definitions}

  Let $( \mathcal{S}, X, \Phi )$ be a dynamical system. For any $x \in X$ we
define the orbit of $x$ as  $O(x)=\{\Phi_s(x): \ s\in \mathcal{S}\} $. A non-empty set $Y \subset X$,  is minimal if $\overline{O(y)}=Y$ for any $y \in Y$.
A point $x \in X$ is minimal if the set $\overline{O(x)}$ is a compact and a minimal set. We denote by $\mathcal{M}$ the set of minimal points. We say that the dynamical system  $( \mathcal{S}, X, \Phi )$
is minimal if there exists a minimal point $x\in X$ such that $\overline{O(x)}=X$.

 The dynamical system  $( \mathcal{S}, X, \Phi )$ is point transitive (PT)
if there exists $x \in X$ such that $\overline{O(x)}=X$, it is  topologically transitive (TT) if given two opene (= open and non-empty) sets
$U,V \subset X$  there exists $s \in \mathcal{S}$ such that $\Phi_s (U) \cap V \neq \emptyset$ and it is densely point transitive (DPT) if
 there exists  a dense set $Y\subset X$ of transitive points. Denote by $Trans(X)$ the set of transitive points.
If $X$ is a Polish space  (i.e. separable completely metrizable topological space) then TT implies DPT.\\
We say that a dynamical system $( \mathcal{S}, X, \Phi )$  is sensitive if there exists $\varepsilon >0$ such that for all $x\in X$ and for all $\delta >0$ there exists $y\in B(x,\delta )$ and
$s \in \mathcal{S}$ such that $d(\Phi_s (x), \Phi_s (y))>\varepsilon $.\\ We will now state our main result:

\begin{theorem}\label{main}
 Assume that the dynamical system $( \mathcal{S}, X, \Phi )$  is TT, non-minimal and $\mathcal{M}$ is a dense set. Then $( \mathcal{S}, X, \Phi )$  is sensitive.
\end{theorem}

We say that the map $\Phi_s$, $s\in \mathcal{S}$, is  almost open if $int (\Phi_s (U))$ ($int(W)$ denotes interior of W) is non-empty whenever $U$ is opene.
The dynamical system $( \mathcal{S}, X, \Phi )$ is almost open if $\Phi_s$ is  almost open for any $s\in \mathcal{S}$.
The almost open dynamical systems include dynamical systems whose elements are open functions or homeomorphisms. When $ X $ is a manifold and the elements of the system are $ C ^{1} $ functions, the set of critical points can be with non-empty interior. But in this case, an open set can not be mapped at a point.

When the system is almost open we obtain a generalization of the Theorem \ref{main}. 

 Denote by $\mathcal{M}^{-1}=\cup_{s \in \mathcal{S}}\Phi_s^{-1}(\mathcal{M})$.

\begin{theorem}\label{teo22}
 Assume that the dynamical system $( \mathcal{S}, X, \Phi )$  is almost open, TT and non-minimal. If  $\mathcal{M}^{-1}$ is dense then $( \mathcal{S}, X, \Phi )$ is sensitive.
\end{theorem}

In section \ref{ejemplos} we construct an example that shows that the above result no longer holds without the hypothesis that $( \mathcal{S}, X, \Phi )$ is almost open.

\section{Proof of Theorems \ref{main} and \ref{teo22}.}

For the proof of the main theorems we need the following lemma:

\begin{lemma}\label{punto_minimal}
   Let $x$ be a minimal point and $U_x$ a neighbourhood of $x$. Then there exist $s_1,...,s_{n_{0}}\in \mathcal{S}$  such that for all  $z\in\overline{O(x)}$  there exists  $i\in \{1,...,n_{0}\}$  with  $\Phi_{ s_{i}}(z)\in U_x$.
\end{lemma}

\begin{proof}
  As $x$ is a minimal point, given  $z\in\overline{O(x)}$  there exist $s_z\in \mathcal{S}$ such that $\Phi_{s_{z}} (z)\in U_x$. By continuity of $\Phi_{s_{z}} $ there exists a
  neighbourhood $V_z$ of $z$ such that $\Phi_{s_{z}}  (V_z)\subset U_x$. As $\overline{O(x)}\subset \cup_{z\in \overline{O(x)}}V_z$, by the compactness of $\overline{O(x)}$, there exist  $z_1,...,z_{n_{0}}\in \overline{O(x)}$ such
  that  $\overline{O(x)}\subset V_{z_{1}}\cup\cdots \cup V_{z_{n_{0}}}$. Then, the functions  $\Phi_{s_{z_{1}}}  ,..., \Phi_{s_{z_{n_{0}} }} $  satisfy the lemma.
\end{proof}

{\bf{Proof of Theorem \ref{main}:}}\\
As the dynamical system $( \mathcal{S}, X, \Phi )$ is non-minimal and $\mathcal{M}$ is dense there exists $M_1$ and $M_2$ minimal sets with $M_1\cap M_2=\emptyset$. Let $8\delta =d(M_1,M_2)$. We will show that the dynamical system $( \mathcal{S}, X, \Phi )$ has sensitive dependence on initial conditions with sensitivity constant $\delta $.
Let $x\in X$ and $U_x$ a neighbourhood of $x$ with $U_x\subset B(x,2\delta)$. Note that $d(x,M_1)\geq 4\delta $ or $d(x,M_2)\geq 4\delta $. Suppose that $d(x,M_1)\geq 4\delta $. As $\mathcal{M}$ is dense there exists $y\in \mathcal{M}\cap U_x$ with $B(y, \delta )\subset B(x,2\delta )$. By Lemma \ref{punto_minimal} there exist
$\Phi_{s_{1}},..., \Phi_{s_{n_{0} }}$ such that for all  $z\in\overline{O(y)}$  there exists $i\in \{1,...,n_{0}\}$  with $\Phi_{s_{_{i}}} (z)\in B(y,\delta)$.\\

Let $W$ be a neighbourhood of $M_1$ such that $$   d( \Phi_{s_{_{i}}} (W),B(x, 2\delta )) \geq \delta \mbox{  for any  } i\in \{1,...,n_{0}\}  .$$
As the dynamical system $( \mathcal{S}, X, \Phi )$ is TT then there exists  $w\in  U_x$ and $\Phi_s$ such that $\Phi_s (w)\in W$. Let $i\in \{1,...,n_{0}\}$ such that $\Phi_{s_{_{i}}}\Phi_s (y)\in B(y,\delta)\subset B(x,2\delta )$. As $\Phi_{s_{_{i}}}\Phi_s (w)\in \Phi_{s_{i}}(W)$ and $  d( \Phi_{s_{_{i}}} (W),B(x, 2\delta )) \geq \delta$, then $d (   \Phi_{s_{_{i}}}\Phi_s (y)      ,  \Phi_{s_{_{i}}}\Phi_s (w)      )\geq \delta $. $\Box$

\vspace{.5cm}

 When $X$ is a Polish space and TT, Theorem \ref{main} generalizes the main result given in \cite{km}.\\

The Lemma \ref{transitive} shows that PT and $\mathcal{M}$ dense implies TT, so we obtain the following corollary:

\begin{corollary}
 Assume that the dynamical system $( \mathcal{S}, X, \Phi )$ is PT and non-minimal. If $\mathcal{M}$ is dense then $( \mathcal{S}, X, \Phi )$ is sensitive.
\end{corollary}

\begin{lemma}\label{transitive}
 Assume that $( \mathcal{S}, X, \Phi )$ is PT and $\mathcal{M}$ is dense. Then $( \mathcal{S}, X, \Phi )$  is TT.

\end{lemma}

\begin{proof}

Let $U$ and $V $ be opene sets and $z\in X$ such that $\overline{O(z)}=X$. Then there exist $\Phi_s, \Phi_{s_{1}} $ and $W_z$ a neighbourhood of $z$ such that $\Phi_s (W_z)\subset U$ and $\Phi_{s_{1}} (W_z)\subset V$.
   As $\mathcal{M}$ is dense there exists $y\in \mathcal{M}\cap W_z$. Then $\Phi_s (y)\in U$ and $\Phi_{s_{1}} (y)\in V$. As $y$ is a minimal point then there exists $\Phi_{s_{2}} $ such that
   $\Phi_{s_{2}}\Phi_s (y)$ is close enough to $\Phi_{s_{1}} (y)$ so that $\Phi_{s_{2}}\Phi_s (y)\in V$. Then $\Phi_{s_{2}}(U)\cap V\neq \emptyset$.
\end{proof}

%
%
%

\subsection{Proof of Theorem \ref{teo22}}


Note that $\mathcal{M}$ is a disjoint union of minimal sets.
 The proof will be divided into two cases.\\
{\bf{ Case (a):}} $\mathcal{M}$ is a minimal set $M$. Let $x_0\in X\setminus M$ and $4\delta =d(x_0,M)$. We will show that the dynamical  system $(\mathcal{S},X,\Phi )$ has sensitive dependence on initial conditions with sensitivity constant $\delta $.

 Let $x\in X $ and $U$ be an open set with $x\in U$. As $\mathcal{M}^{-1}$ is dense there exists $y\in U$ and $\Phi_s$ such that $\Phi_s (y)\in M$.
Let $W_y$ be a neighborhood of $y$, $W_y\subset U$,  such that $d(x_0, \Phi_s (W_y))\geq 3\delta$.

 Since the dynamical  system $(\mathcal{S},X,\Phi )$ is almost open, there exists a open set $V\subset \Phi_s (W_y)$. As the system is TT there exist $w\in V$ and $\Phi_{s_{1}}$ such that $\Phi_{s_{1}} (w)\in B(x_0,\delta )$.
  Note that $w=\Phi_s (z)$ for some $z\in W_y\subset U$ and therefore $\Phi_{s_{1}} (w)=\Phi_{s_{1}} \Phi_s (z)\in B(x_0,\delta )$. As $\Phi_{s_{1}}\Phi_{s}  (y)\in M$, then we have that $y,z\in U$ and \\$d(\Phi_{s_{1}} \Phi_s (y), \Phi_{s_{1}} \Phi_{s}  (z))\geq\delta$.\\
  {\bf {Case (b):}} There exist $M_1,M_2$ minimal sets included in $\mathcal{M}$ with $M_1\neq M_2$. Let $8\delta =d(M_1,M_2)$, $x\in X $ and $U$ be an open set with $x\in U$.
  As $\mathcal{M}^{-1}$ is dense there exist $y\in U$, $\Phi_s$ and a minimal set $M$  such that $\Phi_s (y)\in M$. Note that $d(\Phi_s (y),M_1)\geq 4\delta $ or $d(\Phi_s (y) ,M_2)\geq 4\delta $. Suppose that $d(\Phi_s (y),M_1)\geq 4\delta $.

  Given $B(\Phi_s (y),\delta )$, by Lemma \ref{punto_minimal} there exist
$\Phi_{s_{1}},..., \Phi_{s_{n_{0} }}$ such that for all  $z\in\overline{O(\Phi_s(y))}$  there exists $i\in \{1,...,n_{0}\}$  with $\Phi_{s_{i}}(z)\in B(\Phi_s (y),\delta)$.
Let $W$ a neighbourhood of $M_1$ such that $$   d( \Phi_{s_{i}} (W),B(\Phi_s (y), \delta )) \geq 2\delta \mbox{  for any  } i\in \{1,...,n_{0}\}  .$$

Let $W_y$ be a neighborhood of $y$, $W_y\subset U$,  such that $ \Phi_s (W_y))\subset B(\Phi_s (y),\delta)$.

 Since the dynamical system $(\mathcal{S},X,\Phi)$ is almost open, there exists an opene set $V\subset \Phi_s (W_y)$. As the dynamical system is TT, there exists $w\in V$ and $\Phi_r$ such that $\Phi_r (w)\in W$. Note that $w=\Phi_s (z)$ for some $z\in W_y\subset U$ and $\Phi_r (w)=\Phi_r\Phi_s (z)\in W$.  Let $i\in \{1,...,n_{0}\}$ be such that $\Phi_{s_{i}}\Phi_r \Phi_s (y)\in B(\Phi_s (y),\delta )$.
 As $\Phi_{s_{i}}\Phi_r \Phi_s  (z) \in \Phi_{s_{i}}(W)$ and $  d( \Phi_{s_{i}} (W),B(\Phi_s (y), \delta )) \geq 2\delta$, then $d (   \Phi_{s_{i}}\Phi_r\Phi_s  (z)      , \Phi_{s_{i}}\Phi_r\Phi_s  (y)        )\geq \delta $ and we are done.

\section{Example .}\label{ejemplos}

\begin{figure}[h]
\psfrag{0}{$0$}
\psfrag{1}{$1$}
\psfrag{13}{$\frac{1}{3}$}
\psfrag{23}{$\frac{2}{3}$}
\psfrag{f}{$f_1\equiv 0$}
\psfrag{g}{$f_2$}
\psfrag{12}{$\frac{1}{2}$}
\psfrag{14}{$\frac{1}{4}$}
\psfrag{18}{$\frac{1}{8}$}
\psfrag{h}{$f_3$}
\begin{center}
\caption{\label{figura101}}
\includegraphics[scale=0.22]{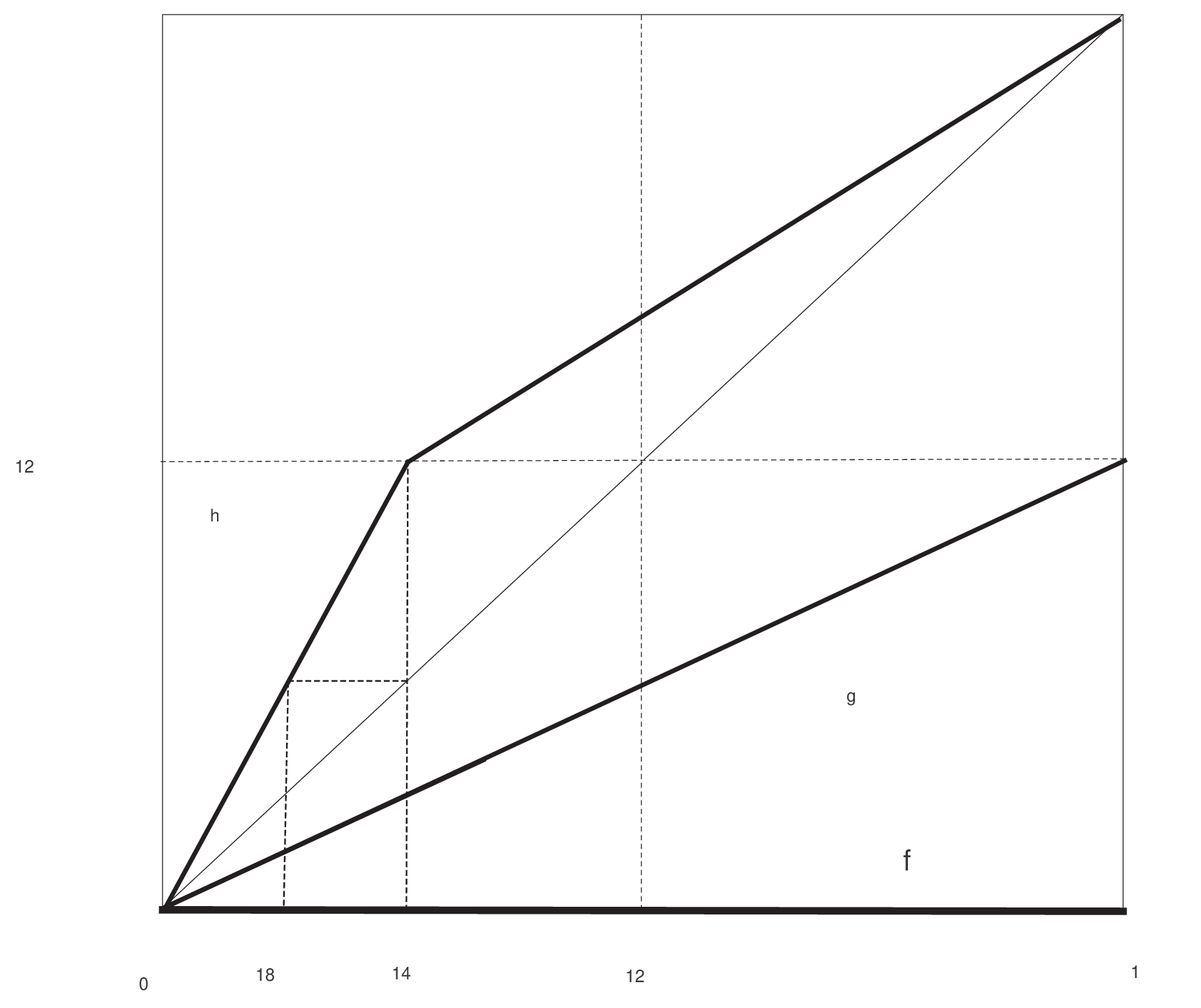}
\end{center}
{ }
\end{figure}

Let $f_1,f_2,f_3:[0,1]\to [0,1]$ be as in Figure \ref{figura101} and  $\mathcal{S}$ the free semigroup generated by three elements $a,b $ and $c$. We consider the action $\Phi$ generated by $\Phi_a=f_1$, $\Phi_b=f_2$ and $\Phi_c=f_3$.
Given $s\in \mathcal{S}$, let $A_{s}=\{x\in (0,1): \ \mbox{ there exists } \Phi_s ^{'}(x) \}$. Note that $A^{c}_{s}$ is a finite set.

 As $f_3|_{[0,\frac{1}{8}]}=f_2^{-1}$ then $f_3\circ f_2=f_2\circ f_3=Id$ in $[0,\frac{1}{8}]$. Since $f^{'}_i|_{(\frac{1}{4},1)}<1$, for all $i=1,2,3.$
 It is not hard to prove that, for any $s\in \mathcal{S}$ and  $x\in A_{s}\cap (\frac{1}{2},1)$,  $\Phi_s ^{'}(x)<1$ (see \cite{ip} ) .
Let $I\subset (\frac{1}{2},1)$  an interval and $s\in \mathcal{S}$. As $A^{c}_{s}$ is a finite set then $$| \Phi_s (I)|=\int_{I\cap A_{s}}\Phi_s ^{'}(x)dx.$$
So, we conclude that $| \Phi_s (I)|\leq |I|$. Therefore the dynamical system $(\mathcal{S},X,\Phi)$ is not sensitive.\\
Now we will prove that for all $x\in (0,1]$, $\overline{O (x)}= [0,1]$.\\
 Let $x\in (0,1]$ and suppose that $\overline{O (x)}\neq [0,1]$. Let $I$ be a connected component of $[0,1]\setminus \overline{O (x)}$ with maximum length.
It is easy to prove that $f_3^{-1}(I)$ or $f_2^{-1}(I)$ or $(f_i f_j)^{-1}(I)$ with $i,j\in \{2,3\}$ has greater length than $I$,  which is a contradiction. Therefore $\overline{O (x)}= [0,1]$.\\
As $\mathcal{M}=\{0\}$ and  $f_1^{-1}(0)=[0,1]=X$, therefore $\mathcal{M}^{-1}$ is dense.
So $(\mathcal{S},X,\Phi)$ is TT and non-minimal, the set $\mathcal{M}^{-1}$ is dense, and $(\mathcal{S},X,\Phi)$ is non-sensitive.


\begin{thebibliography}{99}

\bibitem[AAK]{aak}
\newblock {E. Akin, J. Auslander and K. Berg,}
\newblock When is a transitive map chaotic?,
\newblock \emph{in: Conference in Ergodyc
Theory and Probability (V. Bergelson, K. March and J. Rosenblatt, eds.), de Gruyter, Berlin,}  (1996).



\bibitem[BBCDS]{3}
\newblock Banks, J.; Brooks, J.; Cairns, G.; Davis, G.; Stacey, P.,
\newblock On Devaney's definition of chaos,
\newblock \emph{ Amer. Math. Monthly,} \textbf{99} (1992) , no. 4, 332–-334.


\bibitem[D]{d}
\newblock Dai, Xiongping. ,
\newblock Chaotic dynamics of continuous-time topological semi-flows on Polish spaces,
\newblock \emph{ J. Differential Equations,} \textbf{258} (2015) , no. 8, 2794–-2805.

\bibitem[GRSS]{grss}
\newblock F. H. Ghane, E. Rezaale, M. Saleh and A. Sarizadeh,
\newblock Sensitivity of iterated function systems,
\newblock J. Math. Anal. Appl. 469 (2019), no. 2, 493–-503.


\bibitem[G]{g}
\newblock E. Glasner,
\newblock Ergodic Theory via Joinings. Math. Surveys and Monographs,
\newblock  \emph{vol. 101. Am. Math. Soc., Providence (2003).}
%


\bibitem[IP]{ip}
\newblock Iglesias, J., Portela, A,
\newblock  Almost open semigroup actions,
\newblock \emph{ Semigroup Forum 98 (2019), no. 2, 261–-270}.

\bibitem[KM]{km}
\newblock Kontorovich, E., Megrelishvili, M,
\newblock  A note on sensitivity of semigroup actions,
\newblock \emph{ Semigroup Forum,} \textbf{76} (2008), 133--141.







\end{thebibliography}
\end{document}